\documentclass[12pt]{amsart}
\usepackage{amsmath,latexsym,amscd,amsbsy,amssymb,amsfonts,amsthm,fleqn,leqno,
euscript, graphicx, texdraw, pb-diagram}
\usepackage{mathrsfs}
\usepackage[matrix,arrow,curve]{xy}

\newtheorem{thm}{Theorem}[section]
\newtheorem{pro}[thm]{Proposition}
\newtheorem{lem}[thm]{Lemma}

\newtheorem{cor}[thm]{Corollary}

\def\leukfrac#1/#2{\leavevmode
               \kern.1em
                \raise.9ex\hbox{\the\scriptfont0 ${}_#1$}
                \hskip -1pt\kern-.1em
                /\kern-.15em\lower.10ex\hbox{\the\scriptfont0 ${}_#2$}}

\theoremstyle{definition}
\newtheorem{re}[thm]{Remark}

\theoremstyle{remark}
\newtheorem{claim}{Claim}

\makeatletter
\def\Int{\mathop{\operator@font Int}\nolimits}
\makeatother

\begin{document}

\title[Separation of homogeneous connected locally compact spaces]
{Separation of homogeneous connected locally compact spaces}

\author{Vesko  Valov}
\address{Department of Computer Science and Mathematics, Nipissing University,
100 College Drive, P.O. Box 5002, North Bay, ON, P1B 8L7, Canada}
\email{veskov@nipissingu.ca}
\thanks{The author was partially by NSERC Grant 261914-19}

\keywords{acyclic partitions, cohomological dimension, homogeneous spaces}

\subjclass{Primary 55M10; Secondary 54F45}


\begin{abstract}
We prove that any region $\Gamma$ in a homogeneous $n$-dimensional and locally compact separable metric space $X$, where $n\geq 2$, cannot be irreducibly separated by a closed $(n-1)$-dimensional
subset $C$ with the following property: $C$ is acyclic in dimension $n-1$ and there is a point $b\in C\cap\Gamma$ having a special local base $\mathcal B_C^b$ in $C$ such that the boundary of each $U\in\mathcal B_C^b$ is acyclic in dimension $n-2$. In case $X$ is strongly locally homogeneous, it suffices to have a point $b\in C\cap\Gamma$ with an ordinary base $\mathcal B_C^b$ satisfying the above condition.
The acyclicity means triviality of the corresponding \v{C}ech cohomology groups. This implies all known results concerning the separation of regions in homogeneous connected locally compact spaces.
\end{abstract}

\maketitle

\markboth{}{Separation of homogeneous continua}



\section{Introduction}
By a \emph{space} we mean a locally compact separable metric space, and {\em maps} are continuous mappings. We also consider reduced in dimension zero \v{C}ech cohomology groups $H^n(X;G)$ with coefficients from an Abelian group $G$. If $G$ is the group of the integers $\mathbb Z$, we simply write $H^n(X)$.
Recall that a space $X$ is {\em separated by a set} $C\subset X$ if $C$ is closed in $X$ and $X\backslash C$ is the union of two disjoint open subsets $G_1, G_2$ of $X$. When $C$ is the intersection of the closures $\overline G_1$ and $\overline G_2$, we say that $C$ is an {\em irreducible separator}. A {\em partition} between two disjoint closed sets $A,B$ in $X$ is a closed set $P$ such that $X\backslash P$ is the union of two open disjoint sets $U,V$ in $X$ such that $A\subset U$ and $B\subset V$. 
In such a case we say that $P$ separates $X$ between $A$ and $B$, or $A$ and $B$ are separated in $X$ by $P$. 
A {\em region in $X$} is an open connected subset of $X$.
By $\dim X$ we denote the covering dimension of $X$, and $\dim_GX$ stands for the cohomological dimension of $X$ with respect to a group $G$. 
The boundary of a given set $U\subset X$ in $X$ is denoted by $\rm{bd}U$; if $U\subset C\subset X$, then $\rm{bd}_CU$ denotes the boundary of $U$ in $C$. We say that a point $x\in X$ has a {\em special base} $\mathcal B_x$ if for any neighborhoods $U,V$ of $x$ in $X$ with $\overline U\subset V$ there is
$W\in\mathcal B_x$ such that $\rm{bd}W$ separates $\overline V\backslash U$ between $\rm{bd}\overline V$ and $\rm{bd}\overline U$. 

One of the first results concerning the separation of homogenous metric spaces is the celebrated theorem of Krupski \cite{kru1}, \cite{kru} stating that every region in an $n$-dimensional homogeneous space cannot be separated by a subset of dimension $\leq n-2$. Kallipoliti and Papasoglu \cite{kp} established that any locally connected, simply connected, homogeneous metric continuum cannot be separated by arcs (according to Krupski's theorem, mentioned above, the Kallipoliti-Papasoglu result is interesting for spaces of dimension two). van Mill and the author \cite{vmv} proved that the  Kallipoliti-Papasoglu theorem remains true without simply connectedness, but requiring strong local homogeneity instead of homogeneity.
Recall that a space $X$ is {\em strongly locally homogeneous} if every point $x\in X$ has a base of open neighborhoods $U$ such that for every $y,z\in U$ there is a homeomorphism $h$ on $X$ with $h(y)=z$ and $h$ is the identity on $X\setminus U$. We say that such a homeomorphism $h$ is {\em supported by $U$}.
If for every $x,y\in X$ there is a homeomorphism $h$ on $X$ with $h(x)=y$, the spaces $X$ is {\em homogeneous}. 

In the present paper we establish the following theorem which captures all mentioned above results:
\begin{thm}
Let $\Gamma$ be a region in a homogeneous space $X$ with $dim X=n\geq 2$. Then $\Gamma$ cannot be irreducibly separated by any closed set $C\subset X$ with the following property:
\begin{itemize}
\item[(i)] $\dim C\leq n-1$ and $H^{n-1}(C)=0$;
\item[(ii)] There is a point $b\in C\cap\Gamma$ having a special base $\mathcal B_C^b$ in $C$ with $H^{n-2}(\rm{bd}_CU)=0$ for every $U\in\mathcal B_C^b$.
\end{itemize}
If $X$ is strongly locally homogeneous, condition $(ii)$ can be weakened to the following one:
\begin{itemize}
\item[(iii)]
There is $b\in C\cap\Gamma$ having an ordinary base $\mathcal B_C^b$ in $C$ with $H^{n-2}(\rm{bd}_CU)=0$, $U\in\mathcal B_C^b$.
\end{itemize}
\end{thm}

\begin{re}
According to \cite[Corollary 1.6]{vv}, if $X$ in Theorem 1.1 is a compactum with $H^n(X)\neq 0$, then $X$ is not separated by any $C$ satisfying condition $(i)$.

Since $H^{k+1}(Y)=0$ for
any $k$-dimensional space $Y$, we have the following fact: If $\dim Y\leq n-2$, then $H^{n-1}(Y)=0$ and every $x\in Y$ has a special base of neighborhoods $U$ with $H^{n-2}(\rm{bd}U)=0$ because every two closed subsets of $Y$ can be separated by set $A$ with $\dim A\leq n-3$. Moreover, any $(n-2)$-dimensional separator contains a closed subset which is an irreducible $(n-2)$-dimensional separator. Therefore, Theorem 1.1 implies directly that any region in a homogeneous $n$-dimensional space cannot be separated by a subset of dimension $\leq n-2$. Similar arguments show that if $G$ is any countable Abelian group, then any homogeneous connected space of cohomological dimension $\dim_GX\leq n$ cannot be separated by a closed subset of dimension $\dim_G\leq n-2$ (this fact was established by different methods in \cite{kktv}).

If a region $\Gamma$ in a two-dimensional strongly locally homogeneous space is separated by an arc $C$, then there is a closed $C'\subset C$ irreducibly separating $\Gamma$, see \cite{vmv}. Then $H^1(C')=0$ and the point $b=\max\{x: x\in C'\}$ has a base $\mathcal B_{C'}$ such that $\rm{bd}_{C'}U$ is a point for all $U\in\mathcal B_{C'}$. Therefore, Theorem 1.1 also implies our result \cite{vmv} with van Mill.
\end{re}

Theorem 1.1 is a particular case of the following theorem when $G=\mathbb Z$:
\begin{thm}
Let $\Gamma$ be a region in a finite-dimensional homogeneous space $X$ with $\dim_GX=n\geq 2$, where $G$ is a countable Abelian group. Then $\Gamma$ cannot be irreducibly separated by any closed set $C\subset X$ with the following property:
\begin{itemize}
\item[(i)] $\dim_G C\leq n-1$ and $H^{n-1}(C;G)=0$;
\item[(ii)] There is a point $b\in C\cap\Gamma$ having a special local base $\mathcal B_C^b$ in $C$ with $H^{n-2}(\rm{bd}_CU;G)=0$ 
for every $U\in\mathcal B_C^b$.
\end{itemize}
If $X$ is strongly locally homogeneous, the finite-dimensionality of $X$ can be omitted and condition $(ii)$ can be weakened to the following one:
\begin{itemize}
\item[(iii)]
There is $b\in C\cap\Gamma$ having an ordinary base $\mathcal B_C^b$ in $C$ with $H^{n-2}(\rm{bd}_CU)=0$, $U\in\mathcal B_C^b$.
\end{itemize}
\end{thm}
Theorem 1.3 is established in Section 3. Section 2 contains some definitions and preliminary results. 

\section{Definitions and preliminary results}
 Recall that for any nontrivial Abelian group $G$ the \v{C}ech cohomology group $H^n(X;G)$ is isomorphic to the group $[X,K(G,n)]$ of pointed homotopy classes of maps from $X$ to $K(G,n)$, where $K(G,n)$ is a $CW$-complex of type $(G,n)$, see \cite{hu}.
It is also well known that the circle group $\mathbb S^1$ is a space of type $(\mathbb Z,1)$.
The cohomological dimension $\dim_G(X)$ is the largest number $n$ such that there exists a closed subset $A\subset X$ with $H^n(X,A;G)\neq 0$. Equivalently, for a metric space $X$ we have $\dim_GX\leq n$ if and only if for any closed pair $A\subset B$ in $X$ the homomorphism
$j_{B,A}^n:H^{n}(B;G)\to H^{n}(A;G)$, generated by the inclusion $A\hookrightarrow B$, is surjective, see \cite{dy}. This means that
 every map from $A$ to $K(G,n)$ can be extended over $B$. For every $G$ we have $\dim_GX\leq\dim_{\mathbb Z}X\leq\dim X$, and $\dim_{\mathbb Z}X=\dim X$ in case $\dim X<\infty$ \cite{ku} (on the other hand, there is an infinite-dimensional compactum $X$ with $\dim_{\mathbb Z}X=3$, see \cite{dr}).

Suppose $(K,A)$ is a pair of compact sets in a space $X$ with $\varnothing\neq A\subset K$. We say that $K$ is
an {\em $k$-cohomology membrane spanned on $A$ for an element $\gamma\in H^k(A;G)$} if $\gamma$ is not extendable over $K$,
but it is extendable over every proper closed subset of $K$ containing $A$. Here, $\gamma\in H^k(A;G)$ is {\em extendable over $K$} means that
$\gamma$ is contained in the image $j_{K,A}^k\big(H^{k}(K;G)\big)$. Concerning extendability, we are using the following simple fact:
\begin{lem}
Let $A, B$ be closed sets in $X$ with $X=A\cup B$. Then $\gamma\in H^k(A;G)$ is extendable over $X$ if and only if $j^k_{A,\Gamma}(\gamma)$ is extendable over $B$, where $\Gamma=A\cap B$.
\end{lem}
\begin{proof}
This follows from the Mayer-Vietoris exact sequence
{ $$
\begin{CD}
H^{k}(X;G)@>{\varphi^k}>>H^{k}(A;G)\oplus H^k(B;G)@>{\psi^k}>> H^k(\Gamma;G),\\
\end{CD}
$$}\\
where $\varphi^k(\gamma)=(j^k_{X,A}(\gamma),j^k_{X,B}(\gamma))$ and $\psi^k(\gamma_1,\gamma_2)=j^k_{A,\Gamma}(\gamma_1)-j^k_{B,\Gamma}(\gamma_2)$.
Indeed, suppose $\gamma_\Gamma=j^k_{A,\Gamma}(\gamma)$ is extendable over $B$. So, there is $\alpha\in H^k(B;G)$ with $j^k_{B,\Gamma}(\alpha)=\gamma_\Gamma$. Then, $\psi^k(\gamma,\alpha)=0$, which implies the existence of $\beta\in H^k(X;G)$ such that
$\varphi^k(\beta)=(\gamma,\alpha)$. This yields $j^k_{X,A}(\beta)=\gamma$. Hence, $\gamma$ is extendable over $X$.

To prove the other implication, suppose $j^k_{X,A}(\beta)=\gamma$ for some $\beta\in H^k(X;G)$, and let $\alpha=j^k_{X,B}(\beta)$. Then
$\psi^k(\gamma,\alpha)=0$, which means that $j^k_{B,\Gamma}(\alpha)=j^k_{A,\Gamma}(\gamma)$. Therefore, $j^k_{A,\Gamma}(\gamma)$ is extendable over $B$.
\end{proof}

\begin{lem}\label{membrane} Let $X$ be a homogeneous space with $\dim_GX=n>1$.
 For every $x\in X$ there exists a compactum $M$ containing $x$ such that all sufficiently small neighborhoods $W$ of $x$ in $X$ have the following property:
For every open neighborhood $V$ of $x$ with $\overline V\subset W$  
there exists a nontrivial $\gamma_V\in H^{n-1}(M\cap\rm{bd}\overline V;G)$ and an $(n-1)$-cohomology membrane $K_V\subset M\cap\overline V$ for $\gamma_V$ spanned on $M\cap{\rm bd}\overline V$.
\end{lem}

\begin{proof}
Since $X$ is a countable union of compact sets, there exists a compactum $Y\subset X$ with $\dim_G Y=n$ (otherwise, by the countable sum theorem for $\dim_G$, $\dim_G X\leq n-1$). Since $\dim_G Y=n$ there exists a proper closed subset $F\subset Y$ and  $\gamma\in H^{n-1}(F;G)$ such that $\gamma$ is not extendable over $Y$. Using the continuity of \v{C}ech cohomology \cite{sp}, we can apply Zorn's lemma to conclude there exists a
minimal compact set $M\subset Y$ containing $F$ such that $\gamma$ is not extendable over $M$, but it is extendable over every proper closed subset of $M$ containing $F$. Since $X$ is homogeneous, we can assume that $x\in M\backslash F$.
Now, let show that any neighborhood $W$ of $x$ with $\overline W\subset X\backslash F$ is as required.
Indeed, suppose $V$ is an open neighborhood of $x$ with $\overline V\subset W$.
Then
$M\backslash V$ is a proper closed subset of $M$ containing $F$. Hence, there exists $\gamma'\in H^{n-1}(M\backslash V;G)$ extending $\gamma$ such that $\gamma'$ is not extendable over $M\cap\overline V$. Let $\gamma_V=j_{M\backslash V,M\cap{\rm bd}\overline V}(\gamma')$. 
Observe that $M=(M\backslash V)\cup(M\cap\overline V)$ with $(M\backslash V)\cap((M\cap\overline V)=M\cap{\rm bd}\overline V$. So, by Lemma 2.1,
$M\cap{\rm bd}\overline V$ is nonempty, otherwise $\gamma'$ would be extendable over $M$. By the same reason, $\gamma_V$ is nontrivial and not extendable over $M\cap\overline V$.
Therefore, there is a minimal closed set $K_V\subset M\cap\overline V$ containing  $M\cap{\rm bd}\overline V$ such that $\gamma_V$ is not extendable over $K_V$. Then $K_V$ is an $(n-1)$-cohomology membrane for $\gamma_V$ spanned on $M\cap{\rm bd}\overline V$.
\end{proof}

\begin{pro}\label{pair}
Let $A\subset P$ be a compact pair and $\gamma$ be a nontrivial element of $H^{n-1}(A;G)$.
Suppose there are  closed subsets $P_1,P_2$ of $P$ satisfying the following conditions:
\begin{itemize}
\item $P_1\cup P_2=P$ and $P_1\cap P_2=C\neq\varnothing$;
\item $\gamma$ is extendable over $P_i\cup A$ for each $i=1,2$, but $\gamma$ is not extendable over $P$.
\end{itemize}
Then $H^{n-1}(C,C\cap A;G)\neq 0$.
\end{pro}
\begin{proof}
Consider the commutative diagram below whose rows are parts of Mayer-Vietoris exact sequences, while the columns are exact sequences for the corresponding couples:

{ $$
\begin{CD}
H^{n-1}(P;G)@>{{\varphi^{n-1}_P}}>>H^{n-1}(P_1;G)\oplus H^{n-1}(P_2;G)\\
@ VV{j^{n-1}_{P,A}}V
@VV{j^{n-1}_{P_1\oplus P_2}}V\\
H^{n-1}(A;G)@>{{\varphi^{n-1}_A}}>>H^{n-1}(A\cap P_1;G)\oplus H^{n-1}(A\cap P_2;G)\\
@ VV{\partial_{P,A}}V
@VV{\partial_{P_1\oplus P_2}}V\\
H^{n}(P,A;G)@>{{\varphi^{n}_{P,A}}}>>H^{n}(P_1,P_1\cap A;G)\oplus H^{n}(P_2,P_2\cap A;G)\\
\end{CD}
$$}\\

Here, the maps $j^{n-1}_{P_1\oplus P_2}$ and $\partial_{P_1\oplus P_2}$ are defined by $$j^{n-1}_{P_1,A\cap P_1}\oplus j^{n-1}_{P_2,A\cap P_2},$$
$$\partial_{P_1\oplus P_2}=\partial_{P_1,P_1\cap A}\oplus\partial_{P_2,P_2\cap A}.$$
Recall also that $\varphi^{n-1}_P=(j^{n-1}_{P,P_1},j^{n-1}_{P,P_2})$, the maps $\varphi^{n-1}_A$, $\varphi^{n}_{P,A}$ and $\varphi^{n}_P$ are defined similarly.

Denote $\alpha_i=j^{n-1}_{A,A\cap P_i}(\gamma)$, $i=1,2$.
Since $\gamma$ is extendable over $A\cup P_i$, there exists $\gamma_i\in H^{n-1}(A\cup P_i)$ extending $\gamma$, i.e. $j^{n-1}_{A\cup P_i,A}(\gamma_i)=\gamma$. Let
$\beta_i=j^{n-1}_{A\cup P_i,P_i}(\gamma_i)$. It follows from the Mayer-Vietoris exact sequence
{ $$
\begin{CD}
H^{n-1}(A\cup P_i;G)\to H^{n-1}(A;G)\oplus H^{n-1}(P_i;G)\to H^{n-1}(A\cap P_i;G)\to...\\
\end{CD}
$$}
that $j^{n-1}_{P_i,A\cap P_i}(\beta_i)=\alpha_i$ for every $i=1,2$.
This implies $j^{n-1}_{P_1\oplus P_2}((\beta_1,\beta_2))=(\alpha_1,\alpha_2)$. Since the second column is a part of an exact sequence, the last equality yields $\partial_{P_1\oplus P_2}(\varphi^{n-1}_A(\gamma))=0$. Hence, $\varphi^n_{P,A}(\partial_{P,A}(\gamma))=0$. Note that  $\widetilde\gamma=\partial_{P,A}(\gamma)\neq 0$ because the first column is exact and $\gamma$ is not extendable over $P$.

Finally, since $\varphi^n_{P,A}(\widetilde\gamma)=0$, Proposition \ref{pair} follows from the Mayer-Vietoris exact sequence
{ $$
\begin{CD}
H^{n-1}(C,C\cap A;G)@>{\triangle}>>H^{n}(P,A;G)@>{\varphi^n_{P,A}}>> H^n(P_1,P_1\cap A;G)\oplus  H^n(P_2,P_2\cap A;G).\\
\end{CD}
$$}
\end{proof}

\begin{cor}
Let $X$ be a space with $\dim_GX=n$ and $C\subset X$ be a nonempty separator of $X$. If there is an open set $U$ such that $C\subset U$ and
$\overline U$ is an $(n-1)$-cohomology membrane for some $\gamma\in H^{n-1}(\rm{bd}\overline U;G)$ spanned on $\rm{bd}\overline U$, then
$H^{n-1}(C;G)\neq 0$.
\end{cor}
\begin{proof}
Let $A=\rm{bd}\overline U$ and $P_1,P_2$ be closed subsets of $X$ with $P_1\cap P_2=C$ and $P_1\cup P_2=\overline U$. Since $H^{n-1}(C;G)=H^{n-1}(C,C\cap\rm{bd}\overline U;G)$,
the proof follows from Proposition 2.3
\end{proof}
Corollary 2.4 implies the well known fact \cite{hw} that $H^{n-1}(C;G)\neq 0$ for any compact separator $C$ of $\mathbb R^n$. Indeed, take any ball $\mathbb B^n$ with $C\subset\rm{int}\mathbb B^n$.

  By $X^*= X\cup \{\infty\}$ we denote the one-point compactification of a space $X$. By a {\it metric on $X$ inherited from $X^*$} we mean the  restriction of ${\rm dist}$ to $X$, where ${\rm dist}$ is any metric on $X^*$.

The following version of Effros' theorem \cite{e} is folklore. For the sake of completeness we include a proof.

\begin{thm}\label{effros}
Let $X$ be a homogeneous locally compact space and $\rho$ be a metric on $X^*$. 
Then for any $a\in X$ and $\varepsilon>0$ there exists $\delta>0$ such that for every $x\in X$ with $\rho(x,a)<\delta$ there exists a homeomorphism $h\colon X\to X$ with $h(a)=x$ and $\rho(h(y),y) < \varepsilon$ for all $y\in X$.
\end{thm}

\begin{proof} Let $\mathcal H(X^*)$ be the space of all homeomorphism of $X^*$, endowed with the compact-open topology. Note that $\mathcal H(X^*)$ is a Polish group and its topology is generated by the metric $\hat{\rho}(f,g)=\sup\{\rho(f(x),g(x))\colon x\in X^*\}$. Therefore the set $\mathcal H_X$ consisting of all $h\in\mathcal H(X^*)$ with $h(\infty)=\infty$ is a closed subgroup of $\mathcal H(X^*)$, so $\mathcal H_X$ is also a Polish group. Recall that every homeomorphism $h$ on $X$ can be extended to a homeomorphism $\widetilde h\in\mathcal H_X$, so $\mathcal H_X\neq\varnothing$.
Since the action
$T^*\colon \mathcal H(X^*)\times X^*\to X^*$, $T^* (g,x) = g(x)$, is continuous, so is the action $T\colon \mathcal H_X\times X\to X$, $T(h,x)=h(x)$. Moreover,  $T$ is transitive because $X$ is homogeneous. Hence, by \cite[Theorem 1.1]{vm2}, $T$ is micro-transitive, i.e., for every $x \in X$ and every neighbourhood $U$ of the identity in $\mathcal H_X$ the set $Ux=\{h(x)\colon h\in U\}$ is a neighbourhood of $x$. This implies the statement of the theorem.
  \end{proof}

\section{Proof of Theorem 1.3}

First, consider the case when $X$ is homogeneous. Suppose $C\subset X$ is closed such that $C\cap\Gamma$ irreducibly separates $\Gamma$ and satisfies conditions $(i)-(ii)$. We are going to obtain a contradiction. To this end, fix a metric $\rho$ on $X$ inherited from the one-point compactification $X^*$ of $X$. Then $\Gamma\backslash C=G_1\cup G_2$ with $C'=\overline G_1\cap\overline G_2\cap\Gamma\subset C$, where $G_1, G_2$ are disjoint open subsets of $\Gamma$.
By Lemma 2.2, there exist a compactum $M$ containing $b$ such that all sufficiently small neighborhoods $W$ of $b$ satisfy the thesis of that lemma. We fix such a neighborhood $W$ having a compact closure with
$\overline W\subset\Gamma$. 
\begin{claim}
Following the notations from Lemma 2.2, there exist sufficiently small neighborhoods $V$ of $b$ in $X$ such that $\overline V\subset W$, $\dim_G{\rm bd}\overline V\leq n-1$ and $K_V\backslash\rm{bd}\overline V$ meets both sets $G_1$ and $G_2$.
\end{claim}
Indeed, let $\varepsilon=\min\{\rho(\overline W,X\backslash\Gamma), \rho(b,X\backslash W)\}$ and $\delta>0$ be a number from Theorem 2.5 corresponding to $\varepsilon$ and the point $b$. Take any neighborhood $V$ of $b$ such that $\overline V\subset W$ and the diameter of $V$ is less than $\delta$.  
Since $X$ is finite-dimensional, according to \cite{dk} and \cite{dk1} we can suppose also that  $\dim_G{\rm bd}\overline V\leq n-1$.
Then there is a $\epsilon$-small homeomorphism $h$ on $X$ so that $\overline{h(V)}\subset W$ and $h(b)\in K_V\backslash\rm{bd}\overline V$.
Hence, considering the sets $h^{-1}(V)$ and $h^{-1}(K_V)$ instead of $V$ and $K_V$,
we can assume that $b\in K_V\backslash\rm{bd}\overline V$. Moreover, $\dim_G(K_V\backslash\rm{bd}\overline V)=n$, otherwise $\gamma_V$ would be extendable over $K_V$.

Further, since $\dim_G C\leq n-1$ and $\dim_G(K_V\backslash\rm{bd}\overline V)=n$, $K_V\backslash\rm{bd}\overline V$ is not contained in $C$. So, $K_V\backslash\rm{bd}\overline V$ meets at least one $G_i$, $i=1,2$. If $K_V\backslash\rm{bd}\overline V$ intersects only $G_1$,
then Theorem 2.5 allows us to push $K_V$ towards $G_2$ by a small homeomorphism $h:X\to X$ such that $h(b)\in G_2$, $h(K_V\backslash\rm{bd}\overline V)\cap G_1\neq\varnothing$ and $h(V)$ still contains $b$. This completes the proof of Claim 1.

\begin{claim}
Let $V$ be a neighborhood of $b$ satisfying Claim 1. If $H^{n-2}(C'\cap M\cap\rm{bd}\overline V;G)=0$, we are done.
\end{claim}
Indeed, let $A=M\cap\rm{bd}\overline V$, $P=K_V$ and
$P_i=P\cap\overline G_i$, $i=1,2$. Clearly, Then $P_1\cup P_2=P$ and $P_1\cap P_2=K_V\cap C'$. Hence, 
$A\cap P_1\cap P_2=C'\cap M\cap\rm{bd}\overline V$. Therefore, $\gamma_V$ is a non-trivial element of $H^{n-1}(A;G)$ which is not extendable over $P$. Because $K_V\backslash\rm{bd}\overline V$ meets both sets $G_1$ and $G_2$,
each $A\cup P_i$ is a proper subset of $K_V$ containing  $A$. So, $\gamma_V$
is extendable over each $A\cup P_i$. Therefore, by Proposition 2.3, $H^{n-1}(C',C'\cap A)\neq 0$.
On the other hand, we have the exact sequence
$$
\begin{CD}
H^{n-2}(C'\cap A;G)@>{}>>H^{n-1}(C',C'\cap A;G)@>{}>> H^{n-1}(C';G),\\
\end{CD}
$$
where $H^{n-2}(C'\cap A;G)=0$.
Since $C'$ is a closed subset of $C$, $\dim_GC'\leq n-1$. The last inequality together with $H^{n-1}(C;G)=0$ implies $H^{n-1}(C';G)=0$. Hence, 
$H^{n-1}(C',C'\cap A;G)=0$, a contradiction. This completes the proof of Claim 2.

We use below the following notation: Suppose $\Pi$ is partition in a space $Z$ between two closed disjoint sets $P, Q\subset Z$. Then there are two open disjoint subset $W_P, W_Q$ of $Z$ containing $P$ and $Q$, respectively, such that $Z\backslash\Pi=W_P\cup W_Q$. Then we denote
$\Lambda_P=W_P\cup\Pi$ and $\Lambda_Q=W_Q\cup\Pi$.
\begin{claim}
Suppose that $V$ is a neighborhood of $b$ satisfying Claim 1. 
Then there is another neighborhood $U$ of $b$  with $\overline U\subset V$ such that:
\begin{itemize}
\item[(i)] The element $\gamma_V$ is extendable to an element $\gamma_{V,U}\in H^{n-1}(M(V,U);G)$, where $M(V,U)=\overline V\backslash U$;
\item[(ii)] The element $\gamma_U=j_{M(V,U),{\rm bd}\overline U}(\gamma_{V,U})$ is not extendable over the set $K_U={\rm bd}\overline U\cup(K_V\cap\overline U)$, but  
$\gamma_U$ is extendable over each of the sets $K_{U,i}={\rm bd}\overline U\cup (K_U\cap\overline G_i)$, $i=1,2$;
\item[(iii)] If $\Pi$ separates $M(V,U)$ between ${\rm bd}\overline U$ and ${\rm bd}\overline V$, then there is $\gamma_\Pi\in H^{n-1}(\Pi;G)$ such that $\gamma_\Pi$ is not extendable over $\Lambda_{{\rm bd}\overline U}\cup(K_V\cap\overline U)$, but it is extendable over each set
    $\Lambda_{{\rm bd}\overline U}\cup(K_V\cap\overline U\cap\overline G_i)$, $i=1,2$.
\end{itemize}
\end{claim}
Take points $x_i\in (K_V\backslash\rm{bd}\overline V)\cap G_i$, $i=1,2$. Since 
$K_V$ is an $(n-1)$-cohomology membrane for $\gamma_V$ spanned on $M\cap\rm{bd}\overline V$ and each $K_{V,i}=(M\cap\rm{bd}\overline V)\cup (K_V\cap\overline G_i)$, $i=1,2$, is a proper closed subset of $K_V$ containing $M\cap\rm{bd}\overline V$, $\gamma_V$ can be extended to $\gamma_i\in H^{n-1}(K_{V,i};G)$, $i=1,2$. Using that $\dim_G{\rm bd}\overline V\leq n-1$, we can extend $\gamma_V$ to an element $\gamma_V^*\in H^{n-1}(\rm{bd}\overline V;G)$.
Then $\gamma_V^*$ and $\gamma_i$ provide elements $\gamma_i^*\in H^{n-1}(K_{V,i}^*;G)$, $i=1,2$,
where $K_{V,i}^*=\rm{bd}\overline V\cup(K_V\cap\overline G_i)$, such that $j_{K_{V,1}^*,\rm{bd}\overline V}(\gamma_1^*)=j_{K_{V,2}^*,\rm{bd}\overline V}(\gamma_2^*)=\gamma_V^*$. 

Let $K$ be a $CW$-complex of type $K(G,n-1)$ and the maps $f_V:\rm{bd}\overline V\to K$, $g_{V,i}:K_{V,i}^*\to K$ represent $\gamma_V^*$ and $\gamma_i^*$, respectively, such that both restrictions $g_{V,1}|\rm{bd}\overline V$ and $g_{V,2}|\rm{bd}\overline V$ coincide with $f_V$. Since $G$ is countable, $K$ is also countable and homotopy equivalent to a metrizable simplicial complex. So, we can suppose that $K$ is a metrizable simplicial complex, and let $d$ be a metric on $K$.
Because $K$ is a neighborhood extensor for the class of metrizable spaces, there is an open cover $\omega$ of $K$ such that any two $\omega$-close maps from a given space $Z$ into $K$ are homotopic. Moreover, 
there is an open set $\Omega_i$,  $i=1,2$, in $X$ containing $K_{V,i}^*$ and a map $g_i:\overline\Omega_i\to K$ extending $g_{V,i}$ such that $\overline\Omega_i$ is compact. By the same reason, there is an open set $\Omega_0\subset X$ with a compact closure and a map $g_0:\overline\Omega_0\to K$ extending $f_V$ such that $\rm{bd}\overline V\subset\Omega_0$. We may assume that each $\overline\Omega_i$, $i=0,1,2$, is contained in $W$.
Since $\Theta=\bigcup_{i=0}^{i=2}g_i(\overline\Omega_i)$ is a compact subset of $K$, we can find $\eta>0$ such that any two points $z_1,z_2\in\Theta$ are contained in an element of $\omega$ provided $d(z_1,z_2)<\eta$.  
Then for every $i=0,1,2$ there exists $\delta_i>0$ such that $d(g_i(x),g_i(y))<\eta/2$
for any $x,y\in\overline\Omega_i$ with $\rho(x,y)\leq\delta_i$. Since the points $x_1,x_2$ and $b$ belong to $V\backslash\rm{bd}\overline V$ and 
$\rm{bd}\overline V\subset\Omega_i$ for each $i=0,1,2$, the number
$$\delta=\min\{\delta_i, \rho({\rm bd}\overline V,X\backslash\Omega_i)/2,\rho(b,\rm bd\overline V)/2,\rho(\{x_1\}\cup\{x_2\},\rm bd\overline V)/2:i=0,1,2\}$$ is positive, and let $U=\{x\in V:\rho(x,{\rm bd}\overline V)>\delta\}$. Clearly, $U$ contains the points $x_1,x_2,b$. Moreover $\overline U\subset V$. Indeed, since $\overline U\subset\overline V$, if there is $x\in \overline U\backslash V$ then $x\in \rm bd\overline V$. So, $\rho(x,\rm bd\overline V)=0$ which means that $x\not\in\overline U$. Because $\delta\leq\rho({\rm bd}\overline V,X\backslash\Omega_i)/2$ for each $i=0,1,2$, the set $\overline V\backslash U$ is contained in $\Omega_i$.
Hence, all maps $g_i$ are well defined on $M(V,U)=\overline V\backslash U$.
For every $x\in M(V,U)$ there exists $y\in \rm bd\overline V$ with
$\rho(x,y)\leq\delta$, and since $g_i(y)=f_V(y)$ for all $i$, we have $d(g_i(x),g_j(x))<\eta$ for any $i,j\in\{0,1,2\}$ and $x\in M(V,U)$. This means that for all $x\in M(V,U)$ and $i,j\in\{0,1,2\}$ the points $g_i(x), g_j(x)$ belong to an element of $\omega$.
Therefore, for any closed set $B\subset M(V,U)$ the restrictions $g_{B,i}=g_i|B$,  $i=0,1,2$, are homotopic to each other and represent an element
$\gamma_B\in H^{n-1}(B;G)$. In particular, all $g_{M(V,U),i}$ represent $\gamma_{V,U}\in H^{n-1}(M(V,U);G)$. Similarly, all maps $g_{{\rm bd}\overline U,i}$
represent $\gamma_U\in H^{n-1}({\rm bd}\overline U;G)$. Moreover, since each $g_{M(V,U),i}$ extends $g_{B,i}$, we have $j_{M(V,U),B}(\gamma_{V,U})=\gamma_B$ for all closed $B\subset M(V,U)$. 

So, $j_{M(V,U),{\rm bd}\overline V}(\gamma_{V,U})=\gamma_V^*$ and $j_{M(V,U),{\rm bd}\overline U}(\gamma_{V,U})=\gamma_U$. This means that $\gamma_V^*$ is extendable over $M(V,U)$. Hence, by Lemma 2.1, $\gamma_V^*$ would be extendable over $M(V,U)\cup K_V$ provided $\gamma_U$ is extendable over 
$K_U={\rm bd}\overline U\cup(K_V\cap\overline U)$. In such a case,  
 $\gamma_V$ would be extendable over $K_V$, a contradiction. Therefore, $\gamma_U$ is not extendable over $K_U$.

Consider the sets 
$K_{U,i}={\rm bd}\overline U\cup (K_U\cap\overline G_i)$, $i=1,2$. Each $K_{U,i}$ is a proper closed subset of $K_U$ because so is $K_V\cap\overline U\cap\overline G_i$ in $K_V\cap\overline U$. Observe also that $K_{V,i}^*\cup M(V,U)=M(V,U)\cup (K_V\cap\overline U\cap\overline G_i)$. On the other hand, $\Omega_i$ contains both $K_{V,i}^*$ and $M(V,U)$. So, $\Omega_i$ contains $K_{U,i}$, $i=1,2$. Consequently,
and $g_i|K_{U,i}$ is well defined and extends $g_{{\rm bd}\overline U,i}$. Since $g_{{\rm bd}\overline U,i}$ represents $\gamma_U$, each
$g_i|K_{U,i}$, $i=1,2$, represents an element $\mu_i\in H^{n-1}(K_{U,i};G)$ with $j_{K_{U,i},{\rm bd}\overline U}(\mu_i)=\gamma_U$. This means that $\gamma_U$ is extendable over each of the sets $K_{U,i}$, $i=1,2$.
 
Finally, let $\Pi\subset M(V,U)$ be a closed set separating $M(V,U)$ between ${\rm bd}\overline U$ and ${\rm bd}\overline V$. Then 
$\Lambda_{{\rm bd}\overline V}$ contains both $\Pi$ and ${\rm bd}\overline V$, 
$M(V,U)=\Lambda_{{\rm bd}\overline V}\cup\Lambda_{{\rm bd}\overline U}$ with $\Pi=\Lambda_{{\rm bd}\overline V}\cap\Lambda_{{\rm bd}\overline U}$.
On the other hand, 
$j_{\Lambda_{{\rm bd}\overline V},{\rm bd}\overline V}(\gamma_{\Lambda_{{\rm bd}\overline V}})=\gamma_V^*$ means that $\gamma_V^*$ is extendable over
$\Lambda_{{\rm bd}\overline V}$. Since   
$j_{\Lambda_{{\rm bd}\overline V},\Pi}(\gamma_{\Lambda_{{\rm bd}\overline V}})=\gamma_\Pi$, according to Lemma 2.1, the assumption $\gamma_\Pi$ is extendable over the set 
$\Lambda_{{\rm bd}\overline U}\cup(K_V\cap\overline U)$ would imply that $\gamma_V^*$ is extendable over $M(V,U)\cup K_V$, in particular $\gamma_V$ would be extendable over $K_V$. So, $\gamma_\Pi$ is not extendable over $\Lambda_{{\rm bd}\overline U}\cup(K_V\cap\overline U)$. Let show that $\gamma_\Pi$ is extendable over each
set $\widetilde\Lambda_{{\rm bd}\overline U,i}=\Lambda_{{\rm bd}\overline U}\cup(K_V\cap\overline U\cap\overline G_i)$, $i=1,2$. Observe that $\widetilde\Lambda_{{\rm bd}\overline U,i}$ is contained in $\Omega_i$, $i=1,2$. Consequently, each $h_i=g_i|\widetilde\Lambda_{{\rm bd}\overline U,i}$ is well defined and extends
$g_{\Pi,i}$. On the other hand, all $g_{\Pi,i}$, $i=0,1,2$, are homotopic to each other and represent $\gamma_\Pi$. Hence, each $h_i$, $i=1,2$, represents an element $\nu_i\in H^{n-1}(\widetilde\Lambda_{{\rm bd}\overline U,i};G)$ with $j_{\widetilde\Lambda_{{\rm bd}\overline U,i},\Pi}(\nu_i)=\gamma_\Pi$. Therefore,
$\gamma_\Pi$ is extendable over each $\widetilde\Lambda_{{\rm bd}\overline U,i}$, $i=1,2$. This completes the proof of Claim 3.

\begin{claim}
Suppose $U, V$ are neighborhoods of $b$ satisfying the conditions from Claim 3. If $H^{n-2}(C'\cap K_V\cap\rm{bd}\overline U;G)=0$, we are done.
\end{claim}
Following the notations from Claim 3, denote $A=K_V\cap{\rm bd}\overline U$, $P=K_V\cap\overline U$ and $P_i=K_V\cap\overline U\cap\overline G_i$, $i=1,2$. Then
$P_1\cup P_2=P$, $A\cap P_1\cap P_2=C'\cap K_V\cap{\rm bd}\overline U$. According to Claim 3, $\gamma_U\in H^{n-1}({\rm bd}\overline U;G)$ is not extendable over $K_U={\rm bd}\overline U\cup (K_V\cap\overline U)$. This means that the element 
$\mu_U=j_{{\rm bd}\overline U,K_V\cap{\rm bd}\overline U}(\gamma_U)\in H^{n-1}(A;G)$ is not extendable over $P$. On the other hand, $\gamma_U$ is extendable over each 
$K_{U,i}={\rm bd}\overline U\cup (K_V\cap\overline U\cap\overline G_i)$, $i=1,2$. Consequently, $\mu_U$ is extendable over each $K_V\cap\overline U\cap\overline G_i$. 
To complete the proof of Claim 4, we can apply Proposition 2.3 as we did in Claim 2. This completes the proof of Claim 4.

Therefore, we can suppose everywhere below that there are two neighborhoods $U,V$ of $b$ satisfying the conditions of Claim 3 with $H^{n-2}(C'\cap M\cap\rm{bd}\overline V;G)\neq 0$ and $H^{n-2}(C'\cap K_V\cap\rm{bd}\overline U;G)\neq 0$. In particular, both $C\cap\rm{bd}\overline V$ and 
$C\cap\rm{bd}\overline U$ are nonempty.

\begin{claim}
Let $V,U$ be neighborhoods of $b$ satisfying the conditions from Claim 3 with $C\cap\rm{bd}\overline V\neq\varnothing\neq C\cap\rm{bd}\overline U$.
Then there exists a partition $\Pi$ in $M(V,U)=\overline V\backslash U$ between $\rm{bd}\overline V$ and $\rm{bd}\overline U$ such that $H^{n-2}(\Pi\cap C';G)=0$.
\end{claim}
Consider the set $C\cap M(V,U)$ and its closed disjoint subsets $C\cap\rm{bd}\overline V$ and $C\cap\rm{bd}\overline U$.  
Since $b$ has a special local base $\mathcal B_C^b$ in $C$, there is $W^*\in\mathcal B_C^b$  with $H^{n-2}(\rm{bd}_CW^*;G)=0$
such that
$\rm{bd}_CW^*$ separates $C\cap M(V,U)$ between $C\cap\rm{bd}\overline U$ and $C\cap\rm{bd}\overline V$. 
By \cite[Corollary 3.1.5]{vm}, there exists a partition $T$ in $M(V,V)$ between $\rm{bd}\overline V$ and $\rm{bd}\overline U$ such that $T\cap C\subset\rm{bd}_CW^*$.
Hence, $\Pi=T\cup\rm{bd}_CW^*$ is a partition in $M(U,V)$ between $\rm{bd}\overline V$ and $\rm{bd}\overline U$ with $\Pi\cap C=\rm{bd}_CW^*$ and $H^{n-2}(\Pi\cap C;G)=0$. Finally, since $\Pi\subset\Gamma$, we have $\Pi\cap C=\Pi\cap C'$. This completes the proof of Claim 5.

Now, we can complete the proof of Theorem 1.3 when $X$ is homogeneous. According to Claims 1-5, the proof is reduced to the assumption that there are two neighborhoods $V,U$ of $b$ and a partition $\Pi$ of $M(V,U)$ between $\rm{bd}\overline V$ and $\rm{bd}\overline U$ such that $H^{n-2}(\Pi\cap C';G)=0$.
Then, by Claim 3, there is $\gamma_\Pi\in H^{n-1}(\Pi;G)$ such that $\gamma_\Pi$ is not extendable over $\Lambda_{{\rm bd}\overline U}\cup(K_V\cap\overline U)$, but it is extendable over each set
    $\Lambda_{{\rm bd}\overline U}\cup(K_V\cap\overline U\cap\overline G_i)$, $i=1,2$. In particular, $\gamma_\Pi$ is extendable over each of the sets
$\big(\Lambda_{{\rm bd}\overline U}\cup(K_V\cap\overline U)\big)\cap\overline G_i$.      
    We can apply Proposition 2.3 to obtain a contradiction. Indeed,
denote $P=\Lambda_{{\rm bd}\overline U}\cup(K_V\cap\overline U)$ and $P_i=P\cap\overline G_i$, $i=1,2$. Clearly $P_1\cup P_2=P$ and
$P_1\cap P_2=\big(\Lambda_{{\rm bd}\overline U}\cup(K_V\cap\overline U)\big)\cap C'$. Since $\Pi\subset\Lambda_{{\rm bd}\overline U}$ and $\Pi\cap\overline U=\varnothing$, $P_1\cap P_2\cap\Pi=\Pi\cap C'$. Finally, the exact sequence 
$$
\begin{CD}
H^{n-2}(C'\cap\Pi;G)@>{}>>H^{n-1}(C',C'\cap\Pi;G)@>{}>> H^{n-1}(C';G),\\
\end{CD}
$$
shows that $H^{n-1}(C',C'\cap\Pi;G)=0$ which contradicts Proposition 2.3. Therefore, the homogeneous case of Theorem 1.3 is established.

Consider now the case when $X$ is strongly locally homogeneous. 
\begin{claim}
The point $b$ has a local base in $X$ consisting of open sets $V$ with $H^{n-2}(C'\cap\rm{bd}\overline V;G)=0$.
\end{claim}
Let $W$ be an arbitrary neighborhood of $b$ with $\overline W\subset\Gamma$.
Since, $b$ has a base $\mathcal B_C^b$ in $C$ consisting of sets $U$ with $H^{n-2}(\rm{bd}_CU;G)=0$, there is $U\in\mathcal B_C^b$ such that 
$\overline U\subset W$. Now, we use the following well-known fact \cite{vd}: If $F$ is a closed subset of a metric space $Z$, then there is a correspondence 
$\rm e:\mathcal T(F)\to\mathcal T(Z)$ between the topologies of $F$ and $Z$ such that 
$$\rm e(\Omega)\cap F=\Omega,{~} \rm e(\Omega_1)\cap \rm e(\Omega_2)=\rm e(\Omega_1\cap\Omega_2){~}\mbox{and}{~}  \rm e(\varnothing)=\varnothing.$$
 Such a correspondence is called a {\em $K_0$-function}. It easily seen that 
$\overline{\rm e(\Omega)}\cap F=\overline\Omega$ 
for every open $\Omega\subset F$. 
Now, we consider a $K_0$-function $\rm e:\mathcal T(C'\cap\overline W)\to\mathcal T(\overline W)$ and define $\rm e':\mathcal T(C'\cap W)\to\mathcal T(W)$ by $\rm e'(\Omega)=\rm e(\Omega)\cap W$. Clearly, $\rm e'$ is also a $K_0$-function, and
let $V=\rm e'(U)$. Then $b\in V$ and, according to the mentioned above properties of $K_0$-functions, we have 
$$\overline V\cap C'\subset\overline{\rm e(U)}\cap C'=\overline{\rm e(U)}\cap\overline W\cap C'=\overline U.$$
Since $U\subset\overline V\cap C'$, we obtain  $\overline V\cap C'=\overline U$. Similarly, $V\cap C'=U$.
Moreover, $U\cap\rm{bd}\,\overline{V}=\varnothing$ because 
$U\subset V$. So, $C'\cap\rm{bd}\,\overline{V}\subset\rm{bd}_{C'}U$. On the other hand, $V\cap C'=U$ implies that $\rm{bd}_{C'}U\subset\overline V\backslash V$. Therefore, 
$C'\cap\rm{bd}\,\overline{V}=\rm{bd}_{C'}U$. Clearly, $\rm{bd}_{C'}U=\rm{bd}_{C}U$. So, $H^{n-2}(C'\cap\rm{bd}\overline V;G)=0$.
This completes the proof of Claim 6.


Let $W$ be as Lemma 2.2 and take another two neighborhoods $V,U$ of $b$ such that $\overline U\subset V\subset\overline V\subset W$, $H^{n-2}(C'\cap\rm{bd}\overline V;G)=0$ and for every two points $x,y\in U$ there is a homeomorphism $h$ on $X$ with $h(x)=y$ and $h$ is supported by $U$. According to the proof of Lemma 2.2, the element $\gamma\in H^{n-1}(F;G)$ is extendable to
$\gamma'\in H^{n-1}(M\backslash U;G)$. Let $\gamma_V=j_{M\backslash U,M\cap\rm{bd}\overline V}(\gamma')\in H^{n-1}(M\cap\rm{bd}\overline V;G)$
and $\gamma_U=j_{M\backslash U,M\cap\rm{bd}\overline U}(\gamma')\in H^{n-1}(M\cap\rm{bd}\overline{U};G)$. Then $\gamma_V$ is
not extendable over $M\cap\overline V$ (otherwise $\gamma$ would be extendable over $M$). By the same reason, $\gamma_U$ is not extendable over $M\cap\overline U$. Moreover, $\gamma_V=j_{M\cap(\overline V\backslash U),M\cap\rm{bd}\overline V}(\gamma'')$, where 
$\gamma''=j_{M\backslash U,M\cap(\overline V\backslash U)}(\gamma')$. Hence, $\gamma_V$ is extendable over $M\cap(\overline V\backslash U)$.

Let $A=(C'\cap\rm{bd}\overline V)\cup(M\cap\rm{bd}\overline V)$. Since $\dim_GC'\cap\rm{bd}\overline V\leq\dim_GC'\leq n-1$, there exists $\gamma_A\in H^{n-1}(A;G)$ extending $\gamma_V$. Observe that $A\cap C'=C'\cap\rm{bd}\overline V$ and $A\cap M=M\cap\rm{bd}\overline V$. Since $\gamma_V$ is  extendable
over $M\cap(\overline V\backslash U)$, so is $\gamma_A$ over $(M\cap(\overline V\backslash U))\cup(C'\cap\rm{bd}\overline V)$. On the other hand,
 $\gamma_V$ being not extendable over $M\cap\overline V$ implies
$\gamma_A$ is not extendable over 
$(C'\cap\rm{bd}\overline V)\cup(M\cap\overline V).$
Hence, there is an $(n-1)$-cohomology membrane $K_A\subset(M\cap\overline V)\cup(C'\cap\rm{bd}\overline V)$ for $\gamma_A$ spanned on $A$. 
Since
$\gamma_V$ is extendable over $M\cap(\overline V\backslash U)$, $\gamma_A$ is extendable over $\big(M\cap(\overline V\backslash U)\big)\cup(C'\cap\rm{bd}\overline V)$. Therefore, 
$K_A$ meets $M\cap U$. 
We can suppose that $b\in K_A$. Indeed, if $b\not\in K_A$ take a point $y\in K_A\cap U$ and a homeomorphism $h$ on $X$ supported by $U$ with $h(y)=b$. 
Then $h(K_A)$ is an $(n-1)$-cohomological membrane for $\gamma_A$ spanned on $A$ and contains $b$. 

We can suppose that $\dim_G\rm{bd}\overline U\leq n-1$. Since $K_A\cap\overline U=(K_A\cap\rm{bd}\overline U)\cup(K_A\cap U)$ and $K_A\cap U$ is an $F_\sigma$-set, the assumption $\dim_G K_A\cap U\leq n-1$ would imply that $\dim_GK_A\cap\overline U\leq n-1$ (by the countable sum theorem for $\dim_G$).
Then $\gamma_U$ would be extendable over $K_A\cap\overline U$. Hence, because
$\gamma_A$ is extendable over $(M\cap(\overline V\backslash U))\cup(C'\cap\rm{bd}\overline V)$, we could extend  
$\gamma_A$ over $K_A$. 
Therefore, $\dim_G K_A\cap U=n$. Consequently,  
$K_A\cap U$ meets at least one $G_i$, $i=1,2$. Suppose there is a point $x\in K_A\cap U\cap G_1$. Then $b\neq x$ because $b\not\in G_1$. So, there exists a neighborhood $U'$ of $b$ such that $x\not\in\overline U'$, $\overline U'\subset U$ and 
for every two points $x',y'\in U'$ there is a homeomorphism $h'$ on $X$ supported by $U'$ with $h'(x')=y'$. Since
$U'\cap G_2\neq\varnothing$, we can push $b$ by a homeomorphism $\varphi$ on $X$ supported by $U'$ such that $\varphi(b)\in U'\cap G_2$.
Then $\varphi(K_A)$ is an $(n-1)$-cohomology membrane for $\gamma_A$ spanned on $A$ meeting both $G_i$, $i=1,2$. Hence, $\gamma_A$ is not extendable over $\varphi(K_A)$, but it is extendable over each $\varphi(K_A)\cap\overline G_i$. Finally, let $P=\varphi(K_A)$ and $P_i=\varphi(K_A)\cap\overline G_i$, $i=1,2$. Since $P_1\cup P_2=P$ and $A\cap P_1\cap P_2=A\cap C'$, we can apply Proposition 2.3 with $\gamma=\gamma_A$ to obtain that 
$H^{n-1}(C',A\cap C')\neq 0$. 
Finally, since $H^{n-2}(A\cap C';G)=0$ (recall that $A\cap C'=C'\cap\rm{bd}\overline V$), the exact sequence from the proof of Claim 2 implies $H^{n-1}(C',A\cap C')=0$, a contradiction. This completes the proof of Theorem 1.3.

\textbf{Acknowledgments.} The author would like to express his gratitude to the referee for his/her valuable comments and suggestions.  The author also thanks 
A. Dranishnikov and A. Karassev for their helpful discussions.


\begin{thebibliography}{999}

\bibitem{vd}
E.~K.~ van Dowen, \textit{Simultaneous extensions of continuous functions}, Ph. D. Thesis, Free University of Amsterdam, 1975.

\bibitem{dr}
A.~Dranishnikov, \textit{On a problem of P.S. Alexandrov}, Math. USSSR Sbornik \textbf{63} (1988), 412--426.

\bibitem{dy}
J.~Dydak, \textit{Cohomological dimension and metrizable spaces}, Trans. Amer. Math. Soc. \textbf{337} (1993), 219--234.

\bibitem{dk}
J.~Dydak and A.~Koyama, \textit{Strong cohomological dimension}, Bull. Polish Acad. Sci. \textbf{56} (2008), 193--189.

\bibitem{dk1}
J.~Dydak and A.~Koyama, \textit{Cohomological dimension of locally connected compacta}, Topol. Appl. \textbf{113} (2001), 139--150.

\bibitem{e}
E.~Effros,\textit{Transformation groups and $C^*$-algebras}, Ann. of Math. \textbf{81} (1965), 38--55.

\bibitem{hu}
P.~Huber, \textit{Homotopical cohomology and \v{C}ech cohomology}, Math. Annalen \textbf{144} (1961), 73--76.

\bibitem{hw}
W. Hurewicz and H. Wallman, \textit{Dimension theory}, Princeton University Press, Princeton, 1948.

\bibitem{kp}
M.~Kallipoliti and P.~Papasoglu, \textit{Simply connected homogeneous continua are not separaed by arcs}, Topol. Appl. \textbf{154} (2007),
3039--3047.


\bibitem{kktv}
A.~Karassev, P.~Krupski, V.~Todorov and V.~Valov, \textit{Generalized Cantor manifods and homogeneity}, Houston J. Math. \textbf{38} (2012), 583--609.


\bibitem{kru1}
P.~Krupski, \textit{Recent results on homogeneous curves and $ANR$'s}, Topology Proc. \textbf{16} (1991), 109--118.

\bibitem{kru}
P.~Krupski, \textit{Homogeneity and Cantor manifolds}, Proc. Amer. Math. Soc. \textbf{109} (1990), 1135--1142.


\bibitem{ku}
V.~Kuz'minov, \textit{Homologival dimension theory}, Russian Math. Surveys \textbf{23} (1968), 1--45.


\bibitem{vmv}
J. van Mill and V. Valov, \textit{Homogeneous continua that are are not separated by arcs},
Acta Math. Hung. \textbf{157}, (2019), 364--370.

\bibitem{vm}
J.~van Mill, \textit{The infinite-dimensional topology of function spaces}, North-Holland Publishing Co., Amsterdam, 2001.

\bibitem{vm2}
J.~van Mill, \textit{A note on the Effros theorem} Amer. Math. Monthly 111 (2004), no. 9, 801–806.

\bibitem{sp}
E.~Spanier, \textit{Algebraic Topology}, McGraw-Hill Book Company, 1966.

\bibitem{vv}
V.~Valov, \textit{Homogeneous $ANR$ spaces and Alexandroff manifolds}, Topol. Appl. \textbf{173} (2014), 227--233.

\end{thebibliography}
\end{document}